\documentclass[10pt]{article}
\usepackage{amsmath}
\usepackage{geometry}
\usepackage{amssymb}
\usepackage{setspace}
\usepackage{esint}
\usepackage{amsthm}
\usepackage{extarrows}
\usepackage{mathrsfs}
\usepackage{algorithm}
\usepackage{algorithmic}
\usepackage{subfig}
\usepackage{pgfplots}
\pgfplotsset{compat=1.15}
\usepackage{mathrsfs}
\usetikzlibrary{arrows}
\newtheorem{theorem}{\indent Theorem}[section]
\newtheorem{lemma}[theorem]{\indent Lemma}
\newtheorem{proposition}[theorem]{\indent Proposition}

\newtheorem{definition}{\indent Definition}[section]

\newtheorem{assumption}[theorem]{Assumption}%
\date{\today}
\usepackage{fancyhdr}
\usepackage{lastpage}
\begin{document}
	\title{Fast Accelerated Proximal Gradient Method with New Extrapolation Term for Multiobjective Optimization}
	\author{Huang Chengzhi}
	\maketitle
	\textbf{Abstract:} In this paper, we propose a novel extrapolation coefficient scheme within a new extrapolation term and develop an accelerated proximal gradient algorithm. We establish that the algorithm achieves a sublinear convergence rate. The proposed scheme only requires the Lipschitz constant estimate sequence to satisfy mild initial conditions, under which a key equality property can be derived to support the convergence analysis. Numerical experiments are provided to demonstrate the effectiveness and practical performance of the proposed method.
	
\section{Introduction}\label{sec1}
In practical application scenarios, it is common to encounter problems that involve optimizing multiple objective functions simultaneously. These problems are generally referred to as multiobjective optimization problems, which typically do not have a single optimal value but rather a solution set composed of Pareto-optimal solutions. Consequently, solving such problems presents numerous challenges.

In this paper, we focus on the following so-called composite unconstrained multiobjective optimization problem:
\begin{equation}\label{mop}
	(MOP) \quad \min_{x \in \mathbb{R}^n} F(x) \equiv (f_1(x)+g_1(x), \dotsb, f_m(x)+g_m(x))^T,
\end{equation}
where $f_i :\mathbb{R}^n \to \mathbb{R}$ and $g : \mathbb{R}^n \to \mathbb{R},\forall i = 1,\dotsb,m$ are both convex functions and only $f_i$ is required to be smooth. Many practical problems, such as image restoration and compressed sensing, can be transformed into this form.

A significant advancement in the field was made by Tanabe et al. \cite{tanabe2023accelerated}, who extended the well-known Fast Iterative Shrinkage-Thresholding Algorithm (FISTA) to multi-objective optimization. This extension achieves an impressive convergence rate of $O(1/k^2)$, evaluated using a suitable merit function \cite{tanabe2023new}, marking a notable enhancement over earlier proximal gradient approaches for multi-objective problems (MOP) \cite{tanabe2019proximal}. Building upon this foundation, Nishimura et al. \cite{nishimura2022monotonicity} proposed a monotone variant of the multiobjective FISTA, contributing further to algorithmic stability and practical implementation. In addition, the framework was generalized by Tanabe et al. \cite{tanabe2022globally} through the introduction of hyperparameters, broadening its applicability to include single-objective settings as a special case. Notably, this generalized scheme maintains the favorable $O(1/k^2)$ rate and ensures convergence of the generated sequence, underscoring its robustness and theoretical soundness.


Compared to the extrapolation coefficient used in FISTA, the extrapolation term in the Nesterov-style scheme,
$$
\begin{cases}
	y_{k} = x_{k} + \frac{k -1}{k + \alpha -1} (x_{k} - x_{k-1}) \\
	x_{k+1} = p_{L(f)}(x_k,y_{k})
\end{cases}
$$
 offers more effective acceleration. Sonntag K et al\cite{sonntag2024fast} first implemented this idea in the case of $a = 3$, and Zhang et al.\cite{zhang2023fast} subsequently extended it to the case $a > 3$. However, both works suffer from incomplete theoretical justifications, as their convergence analyses rely on an auxiliary function $\sigma(z) = \min_{i 	=1,\dotsb,m} F_i(x_k) - F_i(z),\forall z \in \mathbb{R}^n$ being nonnegative—a condition that is difficult to guarantee without additional assumptions. To overcome this issue, we propose a modified approach that imposes a mild requirement on the initial estimate of the Lipschitz constant of the gradient of the smooth component $f$ in the objective function, and updates it in a fixed pattern throughout the iterations. This design enables us to derive a key equality that circumvents the need for nonnegativity of $\sigma$, thereby addressing the limitations in previous analyses.

If we remove the smoothness assumption on $f_i $ and consider the case where both $ f_i $ and $ g_i $ are nonsmooth functions but cannot be merged into a single functional form, for example, $ f_i(x) = \max(x, 0) $ and $ g_i(x) = \parallel x \parallel_1 $, then traditional subdifferential based methods, such as Mäkelä et al.’s proximal bundle method \cite{makela2003multiobjective,montonen2018multiple,haarala2004newlimited,makela1992nonsmooth} and the subgradient method, become computationally demanding. This is mainly because the subdifferential is a set-valued mapping rather than a single-valued gradient, making its computation costly and intricate. To address this challenge, Gebken et al. \cite{gebken2021efficient} proposed a subgradient descent algorithm for nonsmooth multi-objective optimization, which combines the descent direction of \cite{mahdavi2012effective} with an approximation based on Goldstein’s $\epsilon$-subdifferential \cite{goldstein1977optimization}.

However, such methods still suffer from the high computational burden associated with evaluating subdifferentials. This motivates a shift toward smoothing techniques \cite{chen2012smoothing}, that is, introducing a smoothing framework to circumvent the complexity of subdifferential computations. By integrating these smoothing strategies with acceleration mechanisms, one can achieve improved convergence rates while maintaining theoretical rigor.

\section{Preliminary Results}\label{Preliminary Results}
In this paper, for any $n \in \mathbb{N}$, $\mathbb{R}^{n}$ denotes the $n$-dimensional Euclidean space. And  $\mathbb{R}^{n}_{+}:= \{ v \in \mathbb{R}^{n} \mid v_{i} \geq 0, i=1,2,\dotsb,n \} \subseteq \mathbb{R}^{n}$ signify the non-negative orthant of $\mathbb{R}^{n}$. Besides, $\Delta^{n} := \{ \lambda \in \mathbb{R}^{n}_{+} \mid \lambda_{i} \geq 0, \sum_{i=1}^{n} \lambda_{i} = 1 \}$ represents the standard simplex in $\mathbb{R}^{n}$. Unless otherwise specified, all inner products in this article are taken in the Euclidean space. Subsequently, the partial orders induced by $\mathbb{R}^{n}_{+}$ are considered, where for any $v^{1}, v^{2} \in \mathbb{R}^{n}$, $v^{1} \leq v^{2}$ (alternatively, $v^{1} \geq v^{2}$) holds if $v^{2} - v^{1} \in \mathbb{R}^{n}_{+}$, and $v^{1} < v^{2}$ (alternatively, $v^{1} > v^{2}$) if $v^{2} - v^{1} \in \text{int} \, \mathbb{R}^{n}_{+}$. In case of misunderstand, we define the order $\preceq(\prec)$ in $\mathbb{R}^{n}$ as $$u\preceq(\prec)v~\Leftrightarrow~v-u\in\mathbb{R}^{n}_{+}(\mathbb{R}^{n}_{++}).$$
Instead, $u \npreceq v$.

From the so-called descent lemma [Proposition A.24 \cite{bertsekas1999nonlinear}], we
have $f_{i}(p)-f_{i}(q)\leq\langle\nabla f_{i}(q),p-q\rangle+(L/2)\|p-q\|^{2}$ for all $p,q\in\mathbb{R}^{n}$ and $i=1,\ldots,m$, which gives

\begin{equation}\label{descent lemma}
	\begin{aligned}F_{i}(p)-F_{i}(r)&=f_{i}(p)-f_{i}(q)+g_{i}(p)+f_{i}(q)-F_{i}(r)\\&\leq\langle\nabla f_{i}(q),p-q\rangle+g_{i}(p)+f_{i}(q)-F_{i}(r)+\frac{L}{2}\left\|p-q\right\|^{2}
	\end{aligned}
\end{equation}
for all $p,q,r\in\mathbb{R}^{n}$ and $i=1,\ldots,m.$

To construct the proximal gradient algorithm, we first introduce some basic definitions for the upcoming discussion. For a closed, proper, and convex function $ h: \mathbb{R}^n \to \mathbb{R} \cup \{ \infty \} $, the Moreau envelope of $ h $ is defined as

$$
\mathcal{M}_h(x) := \min_{y \in \mathbb{R}^n} \left\{ h(y) + \frac{1}{2} \left\| x - y \right\|^2 \right\}.
$$

The unique solution to this problem is called the proximal operator of $ h $, denoted as

$$
\operatorname{prox}_h(x) := \arg\min_{y \in \mathbb{R}^n} \left\{ h(y) + \frac{1}{2} \left\| x - y \right\|^2 \right\}.
$$

Next, we introduce a property between the Moreau envelope and proximal operation by following the lemma.
\begin{lemma}[\cite{rockafellar1997convex}]
	If  {$h$} is a proper closed and convex function, the Moreau envelope
	$\mathcal{M}_{h}$ is Lipschitz continuous and takes the following form,
	$$\nabla \mathcal{M}_{h}(x) := x- prox_{h}(x).$$
\end{lemma}

The following assumption is made throughout the paper:

\begin{assumption}
	$f:\mathbb{R}^n \to \mathbb{R}$ is continuously differentiable with Lipschitz continuous gradient $L(f)$:
	\begin{equation}\label{assump2.1}
		\parallel \nabla f(x) - \nabla f(y) \parallel \leq L(f) \parallel x - y \parallel, \quad \forall x,y \in \mathbb{R}^n,
	\end{equation}
	where $\parallel \cdot \parallel$ stands for standard Euclidean norm unless specified otherwise.
\end{assumption}

We now revisit the optimality criteria for the multiobjective optimization problem denoted as (\ref{mop}). An element $x^{*} \in \mathbb{R}^n$ is deemed weakly Pareto optimal if there does not exist $x \in \mathbb{R}^{n}$ such that $F(x) < F(x^{*})$, where $F: \mathbb{R}^{n} \to \mathbb{R}^{m}$ represents the vector-valued objective function, The ensemble of weakly Pareto optimal solutions is denoted as $X^{*}$. The merit function $u_{0}: \mathbb{R}^{n} \to \mathbb{R} \cup \{\infty\}$, as introduced in \cite{tanabe2023new}, is expressed in the following manner:

\begin{equation}\label{3}
	u_0(x) := \sup_{z \in \mathbb{R}^{n}} \min_{i 	=1,\dotsb,m}[F_{i}(x) - F_{i}(z)].
\end{equation}

The following lemma proves that $u_0$ is a merit function in the Pareto sense.

\begin{lemma}[\cite{tanabe2023accelerated}]
	Let $u_0$ be given as (\ref{3}), then $u_0(x) \geq 0, x \in \mathbb{R}^{n}$, and $x$ is the
	weakly Pareto optimal for (\ref{mop}) if and only if $u_0(x) = 0$.
\end{lemma}

In the context of this study, we introduce an algorithm utilizing the smoothing function delineated in \cite{chen2012smoothing}. This smoothing function serves the purpose of approximating the nonsmooth convex function \(f\) by a set of smooth convex functions, thereby facilitating the application of gradient-based optimization techniques.
\begin{definition}\label{def1}
	For convex function $f$ in (\ref{nmop}), we call $\tilde{f}:\mathbb{R}^{n} \times \mathbb{R}_{+} \to \mathbb{R}$ a smoothing
	function of $f$, if $\tilde{f}$ satisfies the following conditions:
	
	(i) for any fixed $\mu > 0$,$\tilde{f}( \cdot, \mu)$ is continuously differentiable on $\mathbb{R}^{n}$;
	
	(ii) $\lim_{z \to x,\mu \downarrow 0}\tilde{f}(z,\mu) = f(x), \forall x \in \mathbb{R}^{n}$;
	
	(iii) (gradient consistence) $\{\lim_{z \to x,\mu \downarrow 0}\tilde{f}(z,\mu)\} \subseteq \partial f(x), \forall x \in \mathbb{R}^{n}$ ;
	
	(iv) for any fixed $\mu >0$, $\tilde{f}(z,\mu)$ is convex on $\mathbb{R}^{n}$;
	
	(v) there exists a $k > 0$ such that
	$$|\tilde{f}(x,\mu_2) - \tilde{f}(x,\mu_1)| \leq \kappa|\mu_1-\mu_2|, \forall x \in \mathbb{R}^{n}, \mu_1,\mu_2 \in \mathbb{R}_{++};$$
	
	(vi) there exists an $L > 0$ such that $\nabla_{x}\tilde{f}(\cdot,\mu)$ is Lipschitz continuous on $\mathbb{R}^{n}$ with
	factor $L\mu^{-1}$ for any fixed $\mu \in \mathbb{R}_{++}$.
\end{definition}

Combining properties (ii) and (v) in Definition (\ref{def1}), we have

$$|\tilde{f}(x,\mu) - f(x)| \leq \kappa\mu,\forall x \in \mathbb{R}^{n},\mu \in \mathbb{R}_{++}.$$

\section{ An accelerated proximal gradient method for multiobjective optimization}\label{Algorithm}
This section presents an accelerated version of the proximal gradient method. Similar to \cite{tanabe2023accelerated}, we start from (\ref{descent lemma}) and give the subproblem required for each iteration: for a given $x \in dom {F}$, $y \in \mathbb{R}^n$:

\begin{equation}\label{subproblem}
	\min_{z \in \mathbb{R}^n} \varphi_{L(f)} (z; x, y),
\end{equation}
where
$$
\varphi_{L(f)} (z; x, y) := \max_{i=1,\dotsb,n} [ \langle 	\nabla f_i(y), z - y\rangle + g_i(z) + f_i(y) - F_i(x)] + \frac{L(f)}{2} \parallel z - y \parallel^2.
$$

Since $g_i$ is convex for all $i = 1, \dotsb, m$, $z \mapsto \varphi_{L(f)} (z; x, y)$ is strong convex. Thus, the subproblem (\ref{subproblem}) has a unique optimal solution $p_{L(f)}(x, y)$ and takes the optimal function value $\theta_{L(f)} (x, y)$, i.e.,

\begin{equation}\label{pl and theta}
	\begin{aligned}
		p_{L(f)}(x, y) &:= {\arg\min}_{z \in \mathbb{R}^n} 	\varphi_{L(f)}(z; x, y) \\
		\theta_{L(f)}(x, y) &:= {\min}_{z \in \mathbb{R}^n} \varphi_{L(f)}(z; x, y). 	
	\end{aligned}
\end{equation}

Moreover, the optimality condition of (\ref{subproblem}) implies that for all $x \in dom {F}$ and $y \in \mathbb{R}^n$ there exists $\tilde{g}(x, y) \in \partial g(p_{L(f)}(x, y))$ and a Lagrange multiplier $\lambda(x, y) \in \mathbb{R}^m$ such that
\begin{equation}\label{optimal condition a}
	\sum_{i=1}^{m} \lambda_i(x, y) [ \nabla f_i(y) + 	\tilde{g}_i(x, y)] = - L(f) [p_{L(f)}(x, y) - y]
\end{equation}

\begin{equation}\label{optimal condition b}
	\lambda(x, y) \in \Delta^m, \quad \lambda_j (x, y) = 0 	\quad for \ all \ j \notin \mathcal{I}(x, y),
\end{equation}
where $\Delta^m$ denotes the standard simplex and
\begin{equation}
	\mathcal{I}(x, y) := {\arg\min}_{i = 1, \dotsb, m} [ 	\langle \nabla f_i(x), p_{L(f)}(x, y) - y\rangle + g_i(p_{L(f)} (x, y)) + f_i(y) - F_i(x)].
\end{equation}

We note that by taking $z = y$ in the objective function of (\ref{subproblem}), we get
\begin{equation}
	\theta_{L(f)}  (x, y) \leq \varphi_{L(f)} (y; x, y) = 	\max_{i=1,\dotsb,m} \{ F_i(y) - F_i(x)\}
\end{equation}
for all $x \in dom {F}$ and $y \in \mathbb{R}^n$. Moreover, from (8) with $p = z, q = y, r = x$, it follows that
$$
\theta_{L(f)} (x, y) \geq \max_{i=1,\dotsb,m} \{ F_i(p_{L(f)}(x, y)) - F_i(x)\}
$$
for all $x \in dom {F}$ and $y \in \mathbb{R}^n$. Now we use the same statement of [Proposition 4.1, \cite{tanabe2023accelerated}] to characterize weak Pareto optimality in terms of the mappings $p_{L(f)}$ and $\theta_{L(f)}$.

\begin{proposition}\label{prop for pl}
	Let $p_{L(f)}(x, y)$ and $\theta_{L(f)}(x, y)$ be defined by (\ref{pl and theta}). Then, the following three conditions are equivalent:
	
	(i) $y \in \mathbb{R}^n$ is weakly Pareto optimal for (\ref{mop});
	
	(ii) $p_{L(f)} (x, y) = y$ for some $x \in \mathbb{R}^n$;
	
	(iii) $\theta_{L(f)} (x, y) = \max_{i=1,\dotsb,m} [F_i(y) - F_i(x)]$ for some $x \in \mathbb{R}$.
\end{proposition}

\begin{proof}
	The proof is similar to [Proposition 4.1, \cite{tanabe2023accelerated}], so we omit it.
\end{proof}

Proposition \ref{prop for pl} suggests that we can use $\parallel p_{L(f)}(x, y) - y \parallel_{\infty} < \varepsilon$ for some $\varepsilon > 0$ as a stopping criterion.

\begin{algorithm}[H]
	\renewcommand{\algorithmicrequire}{\textbf{Input:}}
	\renewcommand{\algorithmicensure}{\textbf{Output:}}
	\caption{The Smoothing Accelerated Proximal Gradient Method with
		Extrapolation term for Multi-objective
		Optimization}
	\label{alg1}
	\begin{algorithmic}[1]
		\REQUIRE Take initial point $x_{-1}=x_0  \in \text{dom}F$, $y_{0} = x_{0}$, $\varepsilon >0$,  $\gamma_0 \in R_{++}$. Choose parameters  
		$\eta = \frac{(k+\alpha-2)^2}{(k + \alpha -1)(k + \alpha - 3)}$, $0 < \frac{\alpha-2}{\alpha-3}s_0 < \frac{1}{L(f)}$,if $\alpha > 3$, $0 < s_0 < \frac{1}{L(f)}$, if $\alpha = 3$, maximum iteration number $K$.
		\FOR{$k = 0,1, \dotsb, K$}
		\STATE Compute $y_{k} = x_{k} + \frac{k + \alpha - 4}{k + \alpha -1} (x_{k} - x_{k-1}).$
		\STATE compute ${x}_{k+1} = p_{s_k}(x_k,y_{k}).$

		\STATE $s_{k+1} = \eta s_{k}$
		\IF{$\left\|x_k - x_{k+1} \right\| < \varepsilon$}
		\RETURN $x_{k+1}$
		\ENDIF
		\ENDFOR
		\ENSURE $x^*$: A weakly Pareto optimal point
	\end{algorithmic}
\end{algorithm}
The following theorem will explain why we update $s_k$ in this way and why no line search or similar technique is needed to assist the update.

\begin{theorem}
	If $\eta$ is taken in the form of $\eta = \frac{(k+\alpha-2)^2}{(k + \alpha -1)(k + \alpha - 3)}$, and $s_0$ satisfies $0 < \frac{\alpha-2}{\alpha-3}s_0 < \frac{1}{L(f)}$, if $\alpha > 3$; $0 < s_0 < \frac{1}{L(f)}$, if $\alpha = 3$. Then we have:
	$$
	s_k < \frac{1}{L(f)}, \quad \forall k = 1, \dotsb, K.
	$$
\end{theorem}

\begin{proof}
	From the definition of $s_k$, we have 	$s_{k+1} = \frac{(k+ \alpha - 2)^2}{(k+\alpha-1)(k+\alpha-3)} s_k$. 
	When $a > 3$, 
	$$s_{k+1} = \frac{[(k+\alpha-2)!]^2 (\alpha - 2)! (\alpha - 4)!}{[(\alpha - 3)!]^2 (k + \alpha - 1)! (k + \alpha - 3)!} s_0  \leq \frac{k+\alpha -2}{k + \alpha - 1} \frac{1}{L(f)} < \frac{1}{L(f)}.$$ 
	And when $a = 3$, 
	$$
	s_{k+1} = \frac{(k+1)^2}{k(k+2)} s_k = \frac{k+1}{k+2}s_0 < \frac{1}{L(f)}.
	$$
	Therefore, it follows that 
	$$
	s_k < \frac{1}{L(f)}, \quad \forall k = 1, \dotsb, K.
	$$	
\end{proof}
Algorithm \ref{alg1} is known to generate $x_k$ such that $u_0(x_k)$ converges to zero with rate $O(1/k^2)$ under the following assumption.
\begin{assumption}\label{a1}
	Suppose $ X^{*} $ is set of the weakly Pareto optimal points and $\mathcal{L}_{F}(c) := \{x \in \mathbb{R}^{n} F(x) \leq c\}$, then for any $x \in \mathcal{L}_{F}(F(x_0)), k\geq 0$, then there exists
	$x \in X^{*}$such that $F(x^{*}) \leq F(x)$ 
	$$
	\begin{aligned}
		R := \sup_{F^{*} \in F(X^{*} \cap \mathcal{L}_{F}( F(x_0))} &\inf_{z \in F^{-1}(\{F^{*}\})}
		(4 \parallel 2 x_0 - z \parallel^2 + \left\|x^{1}-z\right\|^{2})< \infty.
	\end{aligned}
	$$
\end{assumption}

\section{Convergence Rate of the Algorithm \ref{alg1}}

This section focuses on proving that the algorithm has a convergence rate of $O(1/k^2)$ under Assumption 1. Before doing so, let us introduce some auxiliary sequences in \cite{tanabe2023accelerated} to prove the convergence rate. For $k \geq 0$, let $\sigma_k : \mathbb{R}^n \to \mathbb{R}\cup \{- \infty\}$ and $u_k : \mathbb{R}^n \to \mathbb{R}$ be similarly defined in the following form
\begin{equation}\label{aux seq}
	\begin{aligned}
		\sigma_k(z) &:= \min_{i=1,\dotsb,m} [F_i(x_k) - F_i(z)] \\
		\rho_{p}(z) &:=\left\|(k+\alpha-2)x^{p}-(k + \alpha - 4)x^{p-1} - z\right\|^{2}.
	\end{aligned}
\end{equation}

Now, we give an important property of $\sigma_k(z)$ that will be helpful in the following discussions.

\begin{proposition}\label{prop5.2} 
	
	Let $\sigma_k(z)$ define as (\ref{aux seq}), we have
	\begin{equation}\label{15}
		\sigma_k(z) - \sigma_{k+1}(z)
		\geq -\frac{1}{2s_k} [2\langle  y_k -x_{k+1} , y_k -x_k\rangle  + \parallel x_{k+1} - y_k \parallel^2]. ,
	\end{equation}
	and
	\begin{equation}\label{16}
		\begin{aligned}
			\sigma_{k+1}(z) \leq& \frac{1}{2 s_k}[2\langle y_k-x_{k+1},y_k-z\rangle -\|x_{k+1}-y_k\|^2].
		\end{aligned}
	\end{equation}
\end{proposition}

\begin{proof}
	Recall that there exists $\tilde{g}(x_k, y_k) \in \partial g(x_{k+1})$ and Lagrange multiplier $\lambda_i(x_k, y_k) \in R^m$ that satisfy the KKT condition (\ref{optimal condition a}, \ref{optimal condition b}) for the subproblem (\ref{subproblem}), we have
	$$
	\begin{aligned}
		\sigma_{k+1}(z)   = \min_{i = 1,\dotsb,m} [F_i(x_{k+1}) -  F_i(z)] \leq \sum_{i=1}^{m} \lambda_i(x_k, y_k) [F_i(x_{k+1}) -  F_i(z) ].
	\end{aligned}
	$$
	Essential from the descent lemma,
	$$
	\begin{aligned}
		\sigma_{k+1}(z)  &\leq \sum_{i=1}^{m} \lambda_i(x_{k}, y_k) [F_i(x_{k+1}) -  F_i(z) ] \\
		&\leq \sum_{i=1}^{m} \lambda_i(x_{k}, y_k) [\langle \nabla f_i(y_k ), x_{k+1} - y_k \rangle + g_i(x_{k+1})+ f_i(y_k ) - F_i(z )  ] \\
		&\quad+ \frac{1}{2s_k} \parallel x_{k+1} - y_k \parallel^2.
	\end{aligned}
	$$
	Hence, the convexity of $f_i$ and $g_i$ yields
	$$
	\begin{aligned}
		\sigma_{k+1}(z)  &\leq \sum_{i=1}^{m} \lambda_i(x_{k}, y_k) [\langle \nabla f_i(y_k ) + \tilde{g}(x_k, y_k), x_{k+1} - z \rangle  ] + \frac{1}{2s_k} \parallel x_{k+1} - y_k \parallel^2.
	\end{aligned}
	$$
	Using (\ref{optimal condition a}) with $x = x_k$ and $y = y_k$ and from the fact that $x_{k+1} = p_\ell(x_k, y_k)$, we obtain
	$$
	\begin{aligned}
		\sigma_{k+1}(z)
		&\leq \frac{1 }{2s_k}[2\langle y_k - x_{k+1}, y_k -z \rangle - \parallel x_{k+1} - y_k \parallel^2],
	\end{aligned}
	$$
	
	From the definition of $\sigma_k(z)$, we obtain
	$$
	\sigma_k(z) - \sigma_{k+1}(z)
	\geq - \max_{i =1, \dotsb,m}[F_i(x_{k+1}) - F_i(x_k)] 	
	$$
	Essential from the descent lemma,
	$$
	\begin{aligned}
		&\quad \sigma_k(z)  - \sigma_{k+1}(z)  \\
		&\geq - \max_{i =1, \dotsb,m} [\langle \nabla f_i(y_k), x_{k+1} - y_k \rangle + g_i(x_{k+1}) + f_i(y_k) -F_i(x_k)]- \frac{1}{2s_k} \parallel x_{k+1} - y_k \parallel^2	\\
		&= -\sum_{i=1}^{m} \lambda_i(x_k, y_k) [\langle \nabla f_i(y_k), x_{k+1} - y_k \rangle + g_i(x_{k+1}) + f_i(y_k) -F_i(x_k)] - \frac{1}{2s_k} \parallel x_{k+1} - y_k \parallel^2 \\
		&= -\sum_{i=1}^{m} \lambda_i(x_k, y_k) [\langle \nabla f_i(y_k), x_{k} - y_k \rangle + f_i(y_k) - f_i(x_k)] \\
		&\quad-\sum_{i=1}^{m} \lambda_i(x_k, y_k) [\langle \nabla f_i(y_k), x_{k+1} - x_k \rangle + g_i(x_{k+1}) - g_i(x_k)] - \frac{1}{2s_k} \parallel x_{k+1} - y_k \parallel^2.
	\end{aligned}
	$$
	where the  first equality comes from (\ref{optimal condition b}), and the second one follows by taking $x_{k+1} - y_k = (x_k - y_k) + (x_{k+1} - x_k)$. From the convexity of $f_i, g_i$, we show that
	$$
	\begin{aligned}
		\sigma_k(z)  - \sigma_{k+1}(z)  \geq& -\sum_{i=1}^{m} \lambda_i(x_k, y_k) \langle \nabla f_i(y_k) + \tilde{g}(x_k, y_k), x_{k+1} - x_{k} \rangle - \frac{1}{2s_k} \parallel x_{k+1} - y_k \parallel^2 .	
	\end{aligned}
	$$
	Thus, (\ref{optimal condition a}) and some calculations prove that
	$$
	\begin{aligned}
		\sigma_k(z) - \sigma_{k+1}(z)
		\geq \frac{1}{2s_k} [2\langle x_{k+1} - y_k, y_k -x_k\rangle  - \parallel x_{k+1} - y_k \parallel^2]. \\
	\end{aligned}
	$$
	So we can get two inequalities of $\sigma_{k}(z)$ and $\sigma_{k+1}(z)$ that we want.
\end{proof}

\begin{theorem}\label{fx0}
	Algorithm \ref{alg1} generates a sequence $\{x_k\}$ such that for all $i =1,\dotsb m$ and $k\geq 0$, we have
	$$F_i(x_k) \leq F_i(x_0).$$
\end{theorem}

\begin{proof}
	The proof is similar to [Theorem 5.1 \cite{tanabe2023accelerated}]; we omit it.
\end{proof}

\begin{theorem}\label{convergence1}
	Suppose $\left\{x_k\right\}$ and $\left\{y^k\right\}$ be the sequences generated by algorithm \ref{alg1}, for any $z\in\mathbb{R}^n$, it holds that
	$$
	u_{0}(x_k) \leq \frac{L(f)(\alpha-1)^2}{2(k+\alpha-1)^2}R. 
	$$
\end{theorem}

\begin{proof}
	Let $t_k = \frac{k + \alpha - 3}{k +\alpha-1}$. Multiply (\ref{15}) by $\frac{k + \alpha - 3}{k +\alpha-1}$, and(\ref{16}) by $\frac{2}{k +\alpha-1}$, then adding them, we get:
	\begin{equation}\label{reswkwk1}
		\begin{aligned}
			t_k\sigma_{k}  - \sigma_{k+1} &\geq  \frac{1}{2s_{k+1}}[2\langle x_{k+1}-y_{k}, \frac{k + \alpha - 3}{k +\alpha-1}(y_k - x_k)  +\frac{2}{k +\alpha-1}(y_k - z)\rangle +t_{k}\|x_{k+1}-y_{k}\|^{2}],
		\end{aligned}
	\end{equation}
	
	Recall the definition of  $\rho_{p}$. Then, use the algebraic equality
	\begin{equation}\label{eq abc}
			- \parallel a - b \parallel^2 + 2 \langle b-a, b - c \rangle = - \parallel a - c \parallel^2 + \parallel b - c \parallel^2.
	\end{equation}
	 and the definition of $y_p$:
	\begin{equation}\label{tpwp}
		\begin{aligned}
			s_{p+1}(p + \alpha - 1)(p + \alpha - 3) \sigma_p  - s_{p+1}(p + \alpha - 1)^2\sigma_{p+1}
			&\geq 4(\rho_{p+1} - \rho_p).
		\end{aligned}
	\end{equation}
	
	From the update format of $s_k$, we have	
	$$
	\begin{aligned}
		s_p(p+\alpha-2)^2\sigma_p  -s_{p+1}(p+\alpha-1)^2\sigma_{p+1}
		\geq 4(\rho_{p+1} - \rho_p). \\
	\end{aligned}
	$$
	
	Now, let $k\geq0$. Adding the above inequality from $p=1$ to $p=k$ and using $x_{-1} = x_0$ and $\rho_0 =\left\|(\alpha-2) x^0 - (\alpha - 4)x_{-1}-z\right\|^2  = \parallel 2x_0 - z \parallel^2$, we have
	
	$$
	\begin{aligned}
		s_1 \sigma_1 -s_{k+1}(k+\alpha-1)^2\sigma_{k+1} \geq 4(\rho_k -\left\|2x_0-z\right\|^2).
	\end{aligned}$$
	
	Using
	$$
	-\sigma_{k+1}  \geq \frac{1}{2s_{k+1}}[2\langle x_{k+1}- y_k,y_k-z\rangle  +\|x_{k+1}-y_k\|^2]
	$$
	with $k=0$ and $y^1=x^0$ lead to
	
	\begin{equation}\label{sigma1}
		s_1\sigma_{1} \leq \left\|x^{1}-z\right\|^{2}-\left\|x^{0}-z\right\|^{2} \leq \left\|x^{1}-z\right\|^{2}.
	\end{equation}
	From the above two inequalities, we can derive the desired inequality.
	$$
	\begin{aligned}
		s_{k+1}(k+\alpha-1)^2\sigma_{k+1}
		\leq& s_{k+1}(k+\alpha-1)^2\sigma_{k+1} + \rho_k\\
		\leq&  4 \parallel 2 x_0 - z \parallel^2 + \left\|x^{1}-z\right\|^{2}.
	\end{aligned}
	$$
	
	So from above inequality and Assumption (\ref{a1}), we have
	
	$$
	\sigma_{k+1}
	\leq \frac{R}{s_{k+1}(k+\alpha-1)^2}\leq \frac{L(f)}{(k+\alpha-1)^2}R.
	$$
	
	With similar arguments used in the proof of Theorem 3.1 (see [41, Theorem5.2]), we get the desired inequality:
	$$
	u_{0}(x_k) \leq \frac{L(f)}{(k+\alpha-1)^2}R. 
	$$

\end{proof}

Here, we need to note that the parameter $\alpha$ is taken in the range $\alpha \geq 3$. When $\alpha = 3$, the algorithm degenerates to the case considered in the paper by Sonntag K et al.,\cite{sonntag2024fast} with the following update scheme:
$$
\begin{cases}
	y_{k} = x_{k} + \frac{k + \alpha - 4}{k + \alpha -1} (x_{k} - x_{k-1}) \\
	x_{k+1} = p_{L(f)}(x_k,y_{k})
\end{cases}
$$
This can be recovered from our algorithm simply by setting $\alpha = 3$. It is worth noting that the proof in the work of Sonntag K et al. relies on the nonnegativity of $\sigma_k$. However, it is unfortunate that this property is difficult to establish without certain assumptions.

\section{ An accelerated proximal gradient method for nonsmooth  multiobjective optimization}
This section mainly discuss the following nonsmooth composite problem
:
\begin{equation}\label{nmop}
	(NCMOP) \quad \min_{x \in \mathbb{R}^n} F(x) \equiv (f_1(x)+g_1(x), \dotsb, f_m(x)+g_m(x))^T,
\end{equation}
where $f_i :\mathbb{R}^n \to \mathbb{R}$ and $g : \mathbb{R}^n \to \mathbb{R},\forall i = 1,\dotsb,m$ are both convex nonsmooth functions. And we will extend the algorithm in Section 3 to this case with some modifies. Surprisingly, we find that, if we choose the smoothing parameter $\mu_k$ and the approximate series $\{L_k\}$ under some assumptions, we could get some beautiful results.

Drawing inspiration from the achievements reported in \cite{attouch2016rate}, we incorporate extrapolation techniques with parameters \(\beta_{k} = \frac{k-1}{k + \alpha - 1}\), where \(\alpha > 3\). Choosing the smoothing function \(\tilde{c}\) as defined in Definition (\ref{def1}), we formulate an accelerated proximal gradient algorithm to solve the multiobjective optimization problem denoted as (\ref{nmop}). The algorithm achieves a faster convergence rate while also gain the sequential convergence.

Subsequently, we present the methodology employed to address the optimization problem denoted as (\ref{nmop}). Similar to the exposition in \cite{tanabe2023accelerated}, a subproblem is delineated and resolved in each iteration. Using the descent lemma, the proposed approach tackles the ensuing subproblem for prescribed values of \(x \in \text{dom}(F)\), \(y \in \mathbb{R}^{n}\), and \(\ell \geq L\):
\begin{equation}\label{4}
	\min_{z \in \mathbb{R}^{n}} \varphi_{\ell}(z;x,y,\mu),
\end{equation}
where
\begin{equation}\label{5}
	\varphi_l(z;x,y,\mu) := \max_{i=1,\dotsb,m} \left[\left\langle \nabla \tilde{f}_i (y,\mu),z-y\right\rangle +g_i(z)+\tilde{f}_i(y,\mu)-\tilde{F}_i(x,\mu) \right]  + \frac{\ell }{2} \left\|z-y\right\|^2. \\
\end{equation}

Since $g_{i}$ is convex for all $i = 1,\dotsb,m,z \mapsto \varphi_{\ell}(z;x,y,\mu)$ is strongly convex.Thus,the subproblem (\ref{4}) has a unique optimal solution $p_{\ell}(x,y,\mu)$ and attain the optimal function value $\theta_{\ell}(x,y,\mu)$,i.e.,
\begin{equation}\label{6}
	p_{\ell}(x,y,\mu) := \arg \min_{z\in \mathbb{R}^{n}} \varphi_{\ell}(z,x,y,\mu) \  \text{and} \  \theta_{\ell}(x,y,\mu) := \min_{z\in \mathbb{R}^{n}} \varphi_{\ell}(z,x,y,\mu).
\end{equation}

Furthermore, the optimality condition associated with the optimization problem denoted as (\ref{4}) implies that, for all \(x \in \text{dom} \, F\) and \(y \in \mathbb{R}^{n}\), there exists \(\eta(x, y, \mu) \in \partial g(p_{\ell}(x, y, \mu))\) and a Lagrange multiplier \(\lambda(x, y) \in \mathbb{R}^{m}\) such that
\begin{equation}\label{7}
	\sum_{i=1}^{m} \lambda_{i}(x,y) [ \nabla \tilde{f}_{i}(y,\mu) + \eta_{i}(x,y,\mu)] = -\ell [p_{\ell}(x,y) - y]
\end{equation}
\begin{equation}\label{8}
	\lambda(x,y) \in \Delta^{m}, \quad \lambda_{j}(x,y) = 0 \quad \forall j \notin \mathcal{I}(x,y),
\end{equation}
where $\Delta^{m}$ denotes the standard simplex and
\begin{equation}
	\mathcal{I}(x,y) := \arg\max_{i = 1,\dotsb,m}[ \left\langle \nabla \tilde{f}_{i}(y,\mu) , p_{\ell}(x,y,\mu) - y \right\rangle + g_{i}(p_{\ell}(x,y,\mu)) +\tilde{f}_{i}(y,\mu) - \tilde{F}_{i}(x,\mu) ].
\end{equation}

For easy of reference and corresponding to its structure, we call the proposed
algorithm the smoothing accelerated proximal gradient method with
extrapolation term for nonsmooth multiobjective
optimization(SAPGM) in this paper.The algorithm is in the following form.
\begin{algorithm}[H]
	\renewcommand{\algorithmicrequire}{\textbf{Input:}}
	\renewcommand{\algorithmicensure}{\textbf{Output:}}
	\caption{The Smoothing Accelerated Proximal Gradient Method with
		Extrapolation term for Non-smooth Multi-objective
		Optimization}
	\label{alg2}
	\begin{algorithmic}[2]
		\REQUIRE Take initial point $x_{-1}=x_0  \in \text{dom}F$, $y_{0} = x_{0}$, $\varepsilon >0$, $\mu_0 \in R_{++}$, $s_0 \in R_{++}$.  Choose parameters   $\alpha >3$. Assume $0 < \frac{\alpha - 2}{\alpha - 3}s_0 \mu_0 < \frac{1}{L(f)}$. Set maximum iteration number as $K$.
		\FOR{$k = 0,1, \dotsb, K$}
		
		\STATE Compute $y_{k} = x_{k} + \frac{k + \alpha - 4}{k + \alpha -1} (x_{k} - x_{k-1})$
		
		\STATE Compute ${x}_{k+1} = p_{L_k}(x_k,y_{k},\mu_{k}),$ where $L_k = (s_{k} \mu_{k})^{-1}$.
		
		\IF{$\left\|x_k - x_{k+1} \right\| < \varepsilon$ and $\mu_{k} < \epsilon$}
		\RETURN $x_{k+1}$
		\ENDIF
		
		\STATE Update $\mu_{k+1} = \frac{(k + \alpha -2)\mu_k}{k+ \alpha -1}; \ s_{k+1} = \frac{(k + \alpha - 2) s_k}{k + \alpha -3}.$

		\ENDFOR
		\ENSURE  $x^*$: A weakly Pareto optimal point
	\end{algorithmic}
\end{algorithm}

\section{Convergence rate of the algorithm \ref{alg2}}
This section shows that SAPGM has different convergence rates with different $\sigma$  under the  Assumption (\ref{a1}). For the convenience of the complexity analysis, we use some functions defined in \cite{tanabe2023accelerated}. For $k \geq 0$,for all $z \in \mathbb{R}^n$, let $\tilde{\sigma}_{k} : \mathbb{R}^{n} \to \mathbb{R} \cup \{ - \infty\}$  be defined by
\begin{equation}\label{wkuk}
	\begin{aligned}
		\tilde{\sigma}_{k}(z) &:= \min_{i = 1,\dotsb,m}[\tilde{F}_{i}(x_{k},\mu_{k}) - F_{i}(z)]+ \kappa \mu_{k}. \\
		\tilde{\rho}_{k}(z) &:=\left\|(k+\alpha-2)x^{k}-(k + \alpha - 4)x^{k-1} - z\right\|^{2}.
	\end{aligned}
\end{equation}

\begin{proposition}\label{prop6.2}
	
	Let $\tilde{\sigma}_{k}(z)$ define as (\ref{wkuk}), we have
	\begin{equation}\label{15}
		\tilde{\sigma}_{k+1}(z)\leq \tilde{\sigma}_{k}(z) -\frac{L_{k+1}}{2} [2\langle x_{k+1} - y_k, y_k -x_k\rangle  - \parallel x_{k+1} - y_k \parallel^2],
	\end{equation}
	and
	\begin{equation}\label{16}
		\begin{aligned}
			\tilde{\sigma}_{k+1}(z) \leq& \frac{L_{k+1}}{2}[2\langle y_k-x_{k+1},y_k-z\rangle -\|x_{k+1}-y_k\|^2]+2\kappa\mu_{k+1}.
		\end{aligned}
	\end{equation}
\end{proposition}

\begin{proof}
	Recall that there exists $\eta(x_k, y_k, \mu_{k+1}) \in \partial g(x_{k+1})$ and Lagrange multiplier $\lambda_i(x_k, y_k) \in R^m$ that satisfy the KKT condition (\ref{7},\ref{8}) for the subproblem (\ref{4}), we have
	$$
	\begin{aligned}
		\tilde{\sigma}_{k+1}(z)   &= \min_{i = 1,\dotsb,m} [\tilde{F}_i(x_{k+1}, \mu_{k+1}) -  F_i(z)] + \kappa \mu_{k+1} \\
		&\leq \sum_{i=1}^{m} \lambda_i(x_k, y_k) [\tilde{F}_i(x_{k+1}, \mu_{k+1}) -  F_i(z) + \kappa \mu_{k+1}].
	\end{aligned}
	$$
	Basic from the descent lemma,
	$$
	\begin{aligned}
		\tilde{\sigma}_{k+1}(z)  &\leq \sum_{i=1}^{m} \lambda_i(x_{k}, y_k) [\tilde{F}_i(x_{k+1}, \mu_{k+1}) -  F_i(z) + \kappa \mu_{k+1}] \\
		&= \sum_{i=1}^{m} \lambda_i(x_{k}, y_k) [\tilde{F}_i(x_{k+1}, \mu_{k+1}) - \tilde{F}_i(z, \mu_{k+1}) + \tilde{F}_i(z, \mu_{k+1}) -  F_i(z) + \kappa \mu_{k+1}] \\
		&\leq \sum_{i=1}^{m} \lambda_i(x_{k}, y_k) [\tilde{F}_i(x_{k+1}, \mu_{k+1}) - \tilde{F}_i(z, \mu_{k+1}) + 2\kappa \mu_{k+1}] \\
		&\leq \sum_{i=1}^{m} \lambda_i(x_{k}, y_k) [\langle \nabla \tilde{f}_i(y_k, \mu_{k+1}), x_{k+1} - y_k \rangle + g_i(x_{k+1}) \\
		&\quad + \tilde{f}_i(y_k, \mu_{k+1}) - \tilde{F}_i(z, \mu_{k+1}) + 2\kappa \mu_{k+1}] + \frac{L_k}{2} \parallel x_{k+1} - y_k \parallel^2.
	\end{aligned}
	$$
	Hence,the convexity of $f_i$ and $g_i$ yields
	$$
	\begin{aligned}
		\tilde{\sigma}_{k+1}(z)  &\leq \sum_{i=1}^{m} \lambda_i(x_{k}, y_k) [\langle \nabla \tilde{f}_i(y_k, \mu_{k+1}) + \eta(x_k, y_k, \mu_{k+1}), x_{k+1} - z \rangle + 2\kappa \mu_{k+1}] \\
		& \quad+ \frac{L_k}{2} \parallel x_{k+1} - y_k \parallel^2.
	\end{aligned}
	$$
	Using (\ref{7}) with $x = x_k$ and $y = y_k$ and from the fact that $x_{k+1} = p_\ell(x_k, y_k)$, we obtain
	$$
	\begin{aligned}
		\tilde{\sigma}_{k+1}(z)
		&\leq \frac{L_{k+1}}{2}[2\langle y_k - x_{k+1}, y_k -z \rangle - \parallel x_{k+1} - y_k \parallel^2]+ 2 \kappa \mu_{k+1} ,
	\end{aligned}
	$$
	
	From the definition of $\tilde{\sigma}_{k}(z)$, we obtain
	$$
	\begin{aligned}
		\tilde{\sigma}_{k}(z) - \tilde{\sigma}_{k+1}(z)
		\geq - \max_{i =1, \dotsb,m}[\tilde{F}_i(x_{k+1}, \mu_{k+1}) - \tilde{F}_i(x_k, \mu_k)] + \kappa(\mu_k - \mu_{k+1})	
	\end{aligned}
	$$
	Basic from the descent lemma,
	$$
	\begin{aligned}
		&\quad \tilde{\sigma}_{k}(z)  - \tilde{\sigma}_{k+1}(z)  \\
		&\geq - \max_{i =1, \dotsb,m} [\langle \nabla \tilde{f}_i(y_k, \mu_{k+1}), x_{k+1} - y_k \rangle + g_i(x_{k+1}) + \tilde{f}_i(y_k, \mu_{k+1}) -\tilde{F}_i(x_k, \mu_k)]\\
		&\quad  - \frac{L_{k+1}}{2} \parallel x_{k+1} - y_k \parallel^2+ \kappa(\mu_k - \mu_{k+1})	\\
		&= -\sum_{i=1}^{m} \lambda_i(x_k, y_k) [\langle \nabla \tilde{f}_i(y_k, \mu_{k+1}), x_{k+1} - y_k \rangle + g_i(x_{k+1}) + \tilde{f}_i(y_k, \mu_{k+1}) -\tilde{F}_i(x_k, \mu_k)] \\
		&\quad  - \frac{L_{k+1}}{2} \parallel x_{k+1} - y_k \parallel^2+ \kappa(\mu_k - \mu_{k+1})	 \\
		&= -\sum_{i=1}^{m} \lambda_i(x_k, y_k) [\langle \nabla \tilde{f}_i(y_k, \mu_{k+1}), x_{k} - y_k \rangle + \tilde{f}_i(y_k, \mu_{k+1}) - \tilde{f}_i(x_k, \mu_k)] \\
		&\quad-\sum_{i=1}^{m} \lambda_i(x_k, y_k) [\langle \nabla \tilde{f}_i(y_k, \mu_{k+1}), x_{k+1} - x_k \rangle + g_i(x_{k+1}) - g_i(x_k)] \\
		&\quad- \frac{L_{k+1}}{2} \parallel x_{k+1} - y_k \parallel^2+ \kappa(\mu_k - \mu_{k+1})	.
	\end{aligned}
	$$
	where the  first equality comes from (\ref{8}),and the second one follows by taking $x_{k+1} - y_k = (x_k - y_k) + (x_{k+1} - x_k)$. From Definition \ref{def1} (v) and the convexity of $\tilde{f}_i, g_i$, we show that
	$$
	\begin{aligned}
		\tilde{\sigma}_{k}(z)  - \tilde{\sigma}_{k+1}(z)  \geq& -\sum_{i=1}^{m} \lambda_i(x_k, y_k) \langle \nabla \tilde{f}_i(y_k, \mu_{k+1}) + \eta(x_k, y_k, \mu_{k+1}), x_{k+1} - x_{k} \rangle \\
		& - \frac{L_{k+1}}{2} \parallel x_{k+1} - y_k \parallel^2 .	
	\end{aligned}
	$$
	Thus, (\ref{7}) and some calculations prove that
	$$
	\begin{aligned}
		\tilde{\sigma}_{k}(z) - \tilde{\sigma}_{k+1}(z)
		\geq \frac{L_{k+1}}{2} [2\langle x_{k+1} - y_k, y_k -x_k\rangle  - \parallel x_{k+1} - y_k \parallel^2]. \\
	\end{aligned}
	$$
	So we can get two inequalities of $\tilde{\sigma}_{k}$ and $\tilde{\sigma}_{k+1}$ that we want.
\end{proof}

\begin{theorem}\label{th5.4}
	Suppose $\left\{x_k\right\}$ and $\left\{y^k\right\}$ be the sequences generated by SAPGM, for any $z\in\mathbb{R}^n$, it holds that
	$$
	u_0(x_{k+1}) \leq O\left(\frac{1}{k}\right).
	$$
\end{theorem}

\begin{proof}
	Let$((\ref{15}) \times (\frac{k + \alpha - 3}{k + \alpha -1})+(\ref{16})) \times (\frac{2}{k + \alpha - 1})$, we get:
	\begin{equation}\label{reswkwk1}
		\begin{aligned}
			&\frac{k + \alpha - 3}{k + \alpha - 1} \tilde{\sigma}_{k} (z) - \tilde{\sigma}_{k+1} (z) \\
			\geq& \frac{1}{2 s_{k+1} \mu_{k + 1}} \left[ 2 \left\langle x_{k+1} - y_k, y_k - \frac{k + \alpha - 3}{k + \alpha - 1} x_k - \frac{2}{k + \alpha - 1} z\right\rangle + \parallel x_{k+1} - y_k \parallel\right] \\
			&\quad + \frac{4 \kappa \mu_{k+1}}{k + \alpha - 1}
		\end{aligned}
	\end{equation}
	
	With equality (\ref{eq abc}), we get
	$$
	(k + \alpha - 2)^2 s_{k}\mu_{k} \tilde{\sigma}_{k}(z)  - (k + \alpha - 1)^2 s_{k+1}\mu_{k+1} \tilde{\sigma}_{k+1}(z) \geq 2[\tilde{\rho}_{k+1}(z) - \tilde{\rho}_k(z)] + 4\kappa (k + \alpha - 1) \mu_{k+1}^2 s_{k+1}.
	$$

	Applying this inequality repeatedly, we have
	
	\begin{equation}
		s_{k+1}\mu_{k+1} (k+\alpha-1)^2 \tilde{\sigma}_{k+1}(z) + 2\tilde{\rho}_{k+1}(z) \leq s_1\mu_1 W_{1} + \parallel 2 x_0 - z \parallel^2 + 4\kappa  \sum_{i=1}^{k+1} (i + \alpha - 1) \mu_i^2 s_i
	\end{equation}
	
	And we can find that
	$$
	\begin{aligned}
		\mu_i &= \mu_{i-1} \frac{i + \alpha -2}{i + \alpha -1} = \mu_0 \Pi_{j=1}^{i} \frac{j + \alpha - 2}{j + \alpha - 1} = \mu_0 \frac{\alpha - 1}{i + \alpha - 1}; \\
		s_i &= s_{i-1} \frac{i + \alpha - 2}{i + \alpha - 3} = s_0 \Pi_{j = 1}^{i} \frac{j + \alpha - 2}{j + \alpha - 3} = s_o \frac{i + \alpha - 2}{\alpha - 2},
	\end{aligned}
	$$
	consequently,
	$$
	\begin{aligned}
		4 \kappa (i + \alpha - 1) \mu_i^2 s_i = 4 \kappa (i + \alpha - 1) \frac{(\alpha-1)^2}{(i + \alpha - 1)^2} \mu_0^2 \frac{i + \alpha - 2}{\alpha-2} s_0 \\
		= \frac{i + \alpha -2}{i + \alpha -1} 4\kappa \frac{(\alpha-1)^2}{\alpha-2} \mu_0^2 s_0,
	\end{aligned}
	$$
	we can infer that 
	$$
	\sum_{i=1}^{k+1} 4\kappa (i + \alpha -1) \mu_i^2 s_i \leq O(k).
	$$
	$$
	s_{k+1}\mu_{k+1} (k + \alpha - 1)^2 = \frac{s_0 (k + \alpha -2)}{\alpha - 2} \cdot \frac{\mu_0 (\alpha - 1)}{k + \alpha - 1} \cdot (k + \alpha - 1)^2 = 2(k+\alpha -1)(k+\alpha-2) \leq O(k^2).
	$$
	So, we get that
	$$
	\tilde{\sigma}_{k+1}(z) \leq O\left(\frac{1}{k}\right).
	$$
	which implies that
	$$
	u_0(x_{k+1}) \leq O\left(\frac{1}{k}\right).
	$$
\end{proof}

\section{Numerical Result}
Here we put some numerical results.
\begin{figure}[H]
	\centering
	\subfloat{%
		\includegraphics[width=0.45\textwidth]{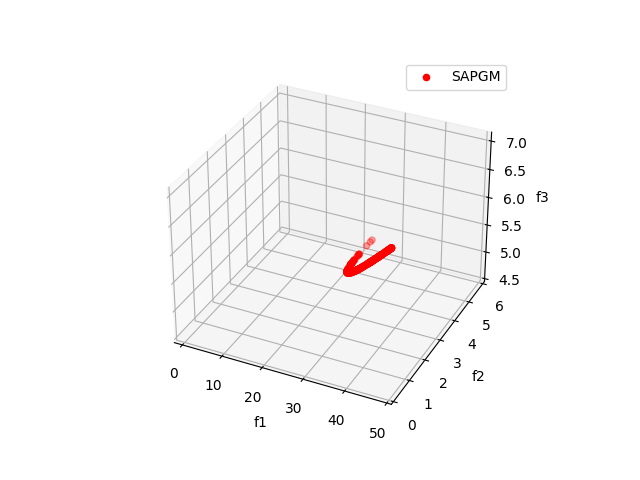}
		\label{BK1SAPGM}
	}
	\hfill
	\subfloat{%
		\includegraphics[width=0.45\textwidth]{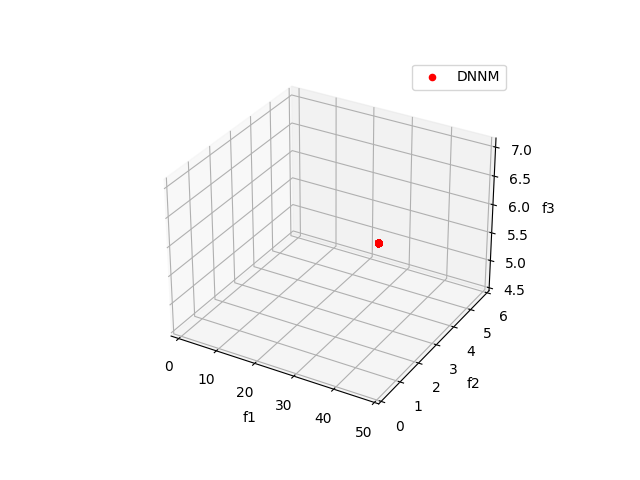}
		\label{BK1DNNM}
	}
	
	\parbox{\textwidth}{\centering (a) BK1\&$\ell_1$}
	
	\subfloat{%
		\includegraphics[width=0.45\textwidth]{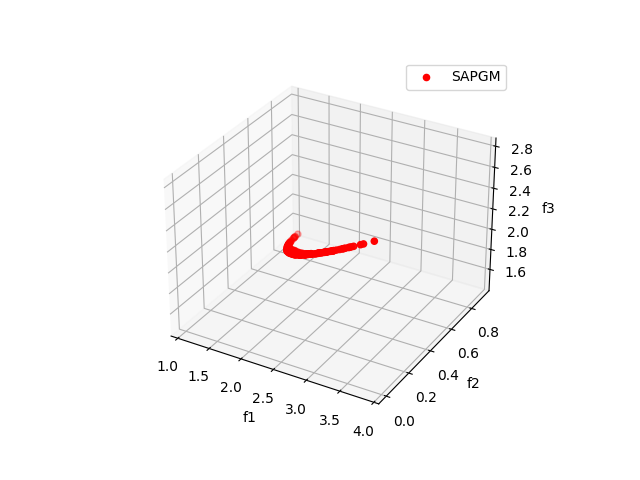}
		\label{jos1SAPGM}
	}
	\hfill
	\subfloat{%
		\includegraphics[width=0.45\textwidth]{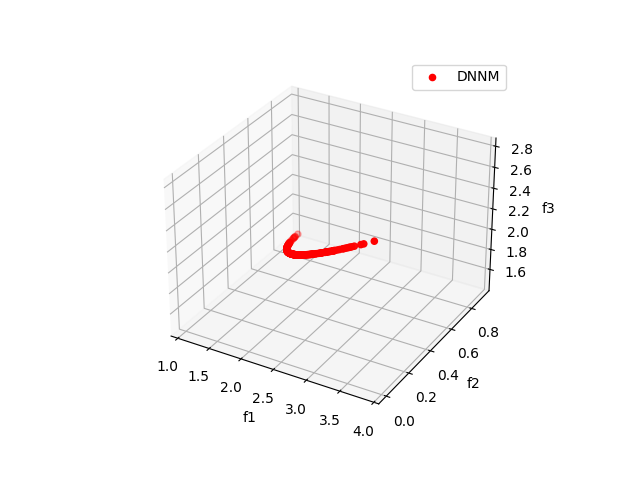}
		\label{fig:jos1DNNM}
	}
	
	\parbox{\textwidth}{\centering (b) JOS1\&$\ell_1$}
	
	\subfloat{%
		\includegraphics[width=0.45\textwidth]{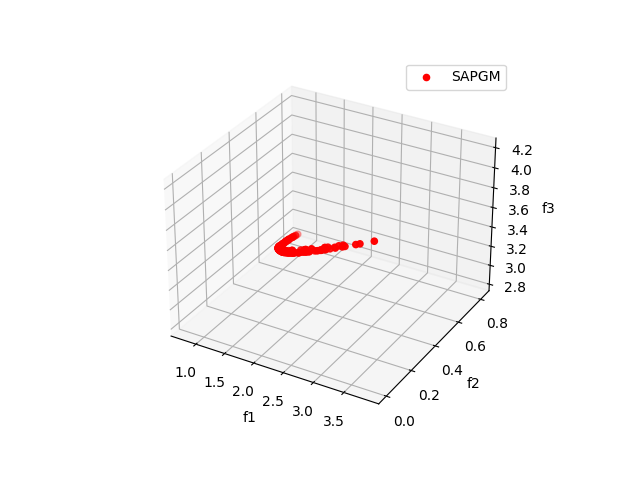}
		\label{SP1SAPGM}
	}
	\hfill
	\subfloat{%
		\includegraphics[width=0.45\textwidth]{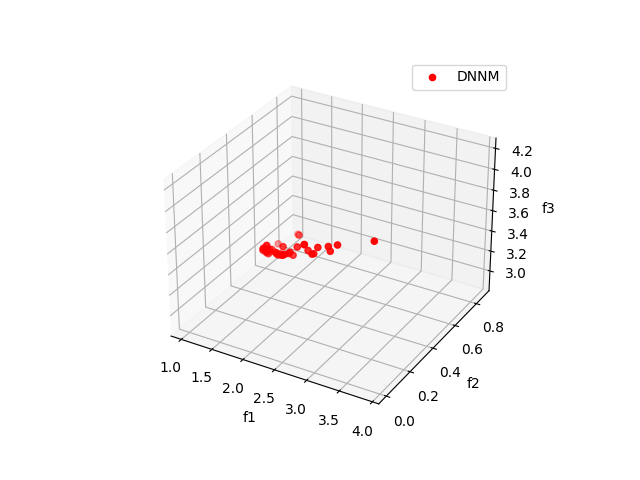}
		\label{SP1DNNM}
	}
	
	\parbox{\textwidth}{\centering (c) SP1\&$\ell_1$}
	
	\caption{The Pareto fronts for Tri-objective optimization problems.}
	\label{The Pareto fronts for Tri-objective optimization problems}
\end{figure} 
	
\end{document}